\documentclass[graybox]{SNmult}

\nonstopmode

\RequirePackage{silence}
\usepackage{type1cm}
\usepackage{makeidx}
\usepackage{graphicx}
\usepackage{multicol}
\usepackage[bottom]{footmisc}
\usepackage{newtxtext}
\usepackage[varvw]{newtxmath}
\usepackage{amsmath,amsfonts,amsthm,amsxtra,bbm}
\usepackage{mathtools}
\usepackage{setspace,graphicx,color,pdflscape}
\usepackage{mathrsfs}

\makeindex
\usepackage[colorlinks=true]{hyperref}
\hypersetup{urlcolor=blue, linkcolor=blue, citecolor=red, anchorcolor=blue]}
\usepackage{orcidlink}

\newcommand{\R}{{\mathbb R}}
\newcommand{\N}{{\mathbb N}}

\newcommand{\be}[1]{\begin{equation}\label{#1}}
\newcommand{\ee}{\end{equation}}
\renewcommand{\(}{\left(}
\renewcommand{\)}{\right)}

\newcommand{\msc}[1]{\href{https://zbmath.org/classification/?q=cc:#1}{#1}}

\newcommand{\sfd}{\mathsf d}
\DeclareMathOperator*{\argmin}{arg\,min}
\newcommand{\Gh}{\mathcal{G}_h}
\newcommand{\Pg}{\mathcal{P}_2(\Gamma)}
\newcommand{\ch}{\mathsf c_h}
\newcommand{\spt}{\mathrm{supp}}
\newcommand{\Fm}{\mathcal{F}_m}
\newcommand{\IFm}{\mathcal{I}_m}
\newcommand{\Cst}{\mathscr C}

\begin{document}

\title*{\Large{Nonlinear kinetic Fokker--Planck equations as gradient flows of the free energy}}
\titlerunning{Kinetic gradient flows}

\author{Giovanni Brigati, Guillaume Carlier, Jean Dolbeault, Filippo Quattrocchi}

\institute{
Giovanni Brigati \orcidlink{0000-0002-3698-9995}\at ISTA, Institute of Science and Technology Austria, Am Campus 1, Klosterneuburg, 3400, Austria, \email{giovanni.brigati@ist.ac.at}
\and
Guillaume Carlier \orcidlink{0000-0001-5801-4622} \at CEREMADE, Centre de Recherche en Math\'ematiques de la D\'ecision (CNRS UMR n$^\circ$~7534),\newline PSL University, Universit\'e Paris-Dauphine, Place de Lattre de Tassigny, 75775, Paris 16, France, and Inria-Paris, Mokaplan, \newline\email{carlier@ceremade.dauphine.fr}
\and
Jean Dolbeault \orcidlink{0000-0003-4234-2298} \at CEREMADE, Centre de Recherche en Math\'ematiques de la D\'ecision (CNRS UMR n$^\circ$~7534),\newline PSL University, Universit\'e Paris-Dauphine, Place de Lattre de Tassigny, 75775, Paris 16, France, \newline\email{dolbeaul@ceremade.dauphine.fr}
\and
Filippo Quattrocchi \orcidlink{0009-0000-9773-1931} \at LMO, Laboratoire de Math\'ematiques d'Orsay, Facult\'e des Sciences d’Orsay, Universit\'e Paris-Saclay, 307 rue Michel Magat, 91405, Orsay, France, \newline\email{filippo.quattrocchi@universite-paris-saclay.fr}
}

\maketitle
\thispagestyle{empty}

\abstract{
We consider nonhomogeneous kinetic equations that involve a free transport operator and a diffusion of porous medium type acting on velocities. The main novelty is a gradient flow interpretation of dynamics driven by an interplay of conservative and dissipative effects. We rely on a notion of discrepancy adapted to a phase space of positions and velocities, built upon second-order characteristics obeying Newton's laws. The equation appears as the steepest descent of the free energy functional. We also prove that approximate solutions constructed with an implicit Euler scheme converge to a solution of the kinetic equation. Thus, we generalise to a family of nonlinear kinetic equations the celebrated JKO scheme in mass transport theory. Most of our results are new even in the case of the linear Vlasov--Fokker--Planck equation.
\\[8pt]
{\bf 2020 MSC\ } Primary: \msc{49Q22}. Secondary: \msc{82C40}; \msc{35Q49}; \msc{35K55}.
\keywords{Kinetic optimal transport $\cdot$ JKO scheme $\cdot$ Gradient flow}
}

\section{Introduction and main results}
\label{Sec:Intro}
For any $m>1$, we consider the nonlinear kinetic Fokker--Planck equation
\be{eq:nlkpm.}
\partial_tf+v\cdot\nabla_xf=\nabla_v\cdot\big(f \, ( \nabla_v P_m(f) + v)\big)
\ee
for $t\ge0$, over the phase space $\Gamma := \R^d \times \R^d\ni(x,v)$ of positions and velocities, where $f$ is a nonnegative distribution function and
\[
P_m(f) \coloneqq \frac{m}{m-1} f^{m-1}
\]
is the pressure variable. With $P_1(f) \coloneqq \log f$, the case $m=1$ is the linear Vlasov--Fokker--Planck equation corresponding to the second-order Langevin process. Equation~\eqref{eq:nlkpm.} admits a Lyapunov functional, namely the \emph{free energy} $\Fm$ given by
\[
\Fm(f) \coloneqq \iint_\Gamma\( \frac1m\,P_m(f) + \frac{1}{2}\,|v|^2 \) \,f dx \, dv\,.
\]
Indeed, for a smooth solution of~\eqref{eq:nlkpm.} with sufficient decay properties, we have
\[
\frac d{dt}\Fm\big(f(t,\cdot,\cdot)\big)=-\,\mathcal{I}_m(f)\coloneqq -\iint_\Gamma \left|\nabla_v P_m(f) + v\right|^2 \, f \, dx \, dv \leq 0\,,
\]
where $\mathcal{I}_m(f)$ denotes a generalized \emph{Fisher information} functional. Characterising the nonhomogeneous kinetic equation~\eqref{eq:nlkpm.} as a gradient flow of the energy $\Fm$ requires a notion of \emph{discrepancy}, tailored on the second-order nature of Newton's equations
\[
x'(t)=v(t)\,, \quad v'(t)=F_t\big(x(t),v(t)\big)
\]
associated to a force field $\R^+ \times \Gamma\ni(t,x,v)\mapsto F_t(x,v)$ with values in $\R^d$. On a finite time interval $[0,T]\ni t$, a curve of measures $(\mu_t)_t$ subject to a measurable force field $(F_t)_t$ with values in $\mathrm{L}^2(\mu_t;\mathbb R^d)$ gets transported according to the Vlasov equation
\begin{equation}
\label{eq:vlasov.}
\partial_t \mu_t + v\cdot \nabla_x \mu_t + \nabla_v \cdot (F_t \, \mu_t) = 0\quad \text{in} \quad\mathcal{D}'\big((0,T) \times \Gamma\big)\,.
\end{equation}
Let $\Pg$ be the set of probability measures $\mu$ on $\Gamma$ with $\iint_\Gamma \(|x|^2+|v|^2 \) d\mu <\infty$. As in~\cite{brigati2025kinetic}, given $\mu$, $\nu \in \Pg$ and $T>0$, we define
\begin{equation}
\label{eq:d_T.}
\mathsf d_T^2(\mu,\nu) \coloneqq \inf_{(\mu_t,F_t)_{t\in[0,T]}} \,T\int_0^T \|F_t\|^2_{\mathrm{L}^2(\mu_t)} \, dt\,,\quad \mu_0=\mu\,, \quad \mu_T = \nu\,,
\end{equation}
by minimising over all admissible $(\mu_t,F_t)_t$ to~\eqref{eq:vlasov.}. In~\cite[Theorems~1.4-1.7]{brigati2025kinetic}, $\sfd_T$-regular curves of measures $(\mu_t)_t$ are identified with solutions to~\eqref{eq:vlasov.} and
\[
|\mu'_t|_\sfd \coloneqq \lim_{h\to0_+} \, \frac{\sfd_{h}(\mu_t,\mu_{t+h})}{h}
\]
is well defined for almost every $t$. We shall write $\Fm(\mu_t)=+\infty$ and $\mathcal{I}_m(\mu_t)=+\infty$ if $\mu_t$ is not absolutely continuous with respect to Lebesgue's measure. For a smooth function~$f$, we have $\nabla_v\cdot\big(f\,\nabla_v P_m(f)\big)=\Delta_vf^m$, $\sqrt f\,\nabla_v P_m(f)=\frac m{m-1/2}\,\nabla_v\big(f^{m-1/2}\big)$
and
\begin{equation}\label{eq:redefinedFisher}
\mathcal{I}_m(f)=\iint_\Gamma\(\big(\tfrac{2\,m}{2\,m-1}\big)^2\,\big|\nabla_v\big(f^{m-1/2}\big)\big|^2\,+|v|^2\,f-2\,d\,f^m\)dx\,dv
\end{equation}
which are relevant quantities for nonsmooth solutions to~\eqref{eq:vlasov.}. Nonnegative weak solutions $f\in\mathrm L^1_{\rm loc}\big(\R^+;\mathrm L^1(1+|v|^2,\Gamma)\big)\cap\mathrm L^m_{\rm loc}\big(\R^+;\mathrm L^m(\Gamma)\big)$ have to be understood in the distributional sense and written in terms of $f$ and $f^m$ only. See Lemma~\ref{lemma:chain} in Section~\ref{Sec:ChainRule-GradienFlow} for details and consequences. With these conventions, we can state our first main result.
\begin{theorem}[Chain rule for the free energy and gradient flow].
\label{thm1.}
Let $m\in[1,3/2]$ and $T>0$. If $\mu_t=f(t,\cdot,\cdot)$ is in $\Pg$ for a.e.~$t$ and solves~\eqref{eq:vlasov.}, with $\int_0^T \big (\Fm(\mu_t)+ \mathcal{I}_m(\mu_t)\big) dt + \int_0^T \|F_t\|^2_{\mathrm{L}^2(\mu_t)} \, dt< \infty$, then, for a negligible set $A\subset[0,T]$ and a.e.~$(a,b)\in\big([0,T]\setminus A\big)^2$ with $b>a$, 
\begin{equation}
\label{eq:chainp.}
\begin{aligned}
\Fm(\mu_b) - \Fm(\mu_a) &= \int_a^b \iint_\Gamma F_t \cdot \big(\nabla_v P_m(\mu_t) + v\big) \, \mu_t \, dx\,dv \, dt\\
&\geq -\,\frac{1}{2} \int_a^b \(|\mu'_t|_\sfd^2+\mathcal{I}_m(\mu_t)\) dt\,.
\end{aligned}
\end{equation}
Equality in the inequality of~\eqref{eq:chainp.} is attained if and only if $f$ is a solution to the nonlinear kinetic Fokker--Planck equation~\eqref{eq:nlkpm.} on the time interval $[0,T]$. Moreover, in this case, \begin{equation} \label{eq:metder} |\mu'_t|_\sfd = \mathcal{I}_m(\mu_t) \, .
\end{equation}
\end{theorem}
\noindent In other words, among solutions of Vlasov's equations, all curves of maximal energy dissipation (Fisher information)  solve~\eqref{eq:nlkpm.}. Our proof crucially relies on convexity properties of the Fisher information, which are only ensured under the condition~$m \in [1,3/2]$. Next, let us consider a second-order version of the JKO scheme.
\begin{theorem}[Minimising-movement scheme and maximal dissipation].\label{thm2.}\\
Let $m \geq 1.$ The scheme starting at $\mu^{(h)}_0 = \mu_0\in\Pg$ with $\Fm(\mu_0)<\infty$, and
\begin{equation}
\label{eq:JKO.}
\mu^{(h)}_{k+1} \in \argmin_{\nu \in \Pg} \( \Fm(\nu) + \frac{1}{2\,h} \, \sfd_h^2\Big(\mu^{(h)}_k,\nu\Big)\)
\end{equation}
is well-posed for any $h \in (0,1)$ and $k\in\N$.
The piecewise constant interpolations $\big(\mu^{(h)}_t\big)_t$
are relatively compact in~$\mathrm{L}^1_\mathrm{loc}\big(\R^+;\mathrm{L}^1(\Gamma)\big)$. Up to the extraction of a sequence, any limit $\mu = \lim_{h \to 0_+} \mu^{(h)}$ solves~\eqref{eq:nlkpm.} in $\mathcal{D}'\big(\R^{+} \times \Gamma\big)$ with initial datum~$\mu_0$. Moreover, if $m\in [1, 3/2]$, $\mu$ dissipates the energy maximally, in the sense that~\eqref{eq:chainp.} is an equality.
\end{theorem}
One crucial point of our scheme~\eqref{eq:JKO.} is \emph{strong compactness}, which we establish by extending the $\mathrm{L}^1$-contraction argument of~\cite{MR4553540} (also see~\cite{leger2025syntheticapproachcomparisonprinciples}) beyond the case of bounded domains. Contraction entails uniform equi-continuity with respect to $x$-translations, while the Fisher information $\mathcal{I}_m$ bounds the oscillations in the $v$-variable. Thus, we control the relevant nonlinear quantities along the scheme and identify limit curves $\mu$ with solutions to~\eqref{eq:nlkpm.}. The threshold $m =3/2$ does not play a role here, as $\mathcal{I}_m$ is lower semicontinuous in the strong topology for any $m \geq 1$.

Moreover, we prove that limit curves produced by the scheme satisfy some reverse form of Inequality~\eqref{eq:chainp.} for any $m\ge1$, see~\eqref{eq:ede3}. When $m\in[1,3/2]$, the maximal dissipation property of Theorem~\ref{thm1.} follows, \emph{i.e.}, we have equality in~\eqref{eq:chainp.}. 
This gradient flow interpretation of~\eqref{eq:nlkpm.} is new, even in the case $m=1$, and the strong compactness for~\eqref{eq:JKO.} may be of independent interest.

Equation~\eqref{eq:nlkpm.} is probably an oversimplified nonlinear kinetic equation for significant applications in physics. However, the gradient flow property not only makes sense in a functional framework of mathematical interest, but also admits an interpretation in mechanics that may apply in more realistic models.

\medskip Let us conclude this introduction with a brief review of the literature. The \emph{JKO scheme}~\cite{MR1617171} (also see~\cite[Section~4.4]{MR3625852}) characterises the linear Fokker--Planck equation as the maximal dissipation flow of the free energy according to the Wasserstein distance. This property was formally extended to nonlinear diffusions by \emph{Otto's calculus} in~\cite{MR1842429} and later justified as a gradient flow in~\cite{MR2401600}. In linear diffusion models, adapted free energy functionals and distances can be chosen simultaneously~\cite{MR2448650}. \emph{Minimising-movement schemes for kinetic equations} appeared in~\cite{MR1893894,MR1763126}, and later in~\cite{MR2731396,MR3271101,MR3129075}. For the classical alternative of splitting schemes, we refer to~\cite{park2024variational} and references therein. Various notions of optimal transport distances adapted to Newton's laws have been considered for instance in~\cite{MR1763126,MR1893894,MR2560310,MR2731396,brigati2025kinetic,lee2026deepkineticjkoschemes}, which also play a role in \emph{Wasserstein stability} of solutions to Vlasov's equation, including mean-field forces, Poisson kernels and Fokker--Planck diffusions in velocities (see, \emph{e.g.},~\cite{MR4393384} and references therein). Such an approach has interesting consequences for large time asymptotics, uniqueness, particle approximation and propagation of molecular chaos, which are out of our scope. However, we can refer to~\cite{lee2026deepkineticjkoschemes} for a recent numerical implementation with $m=1$, with or without self-consistent mean field Poisson coupling. In this paper, we rely on the \emph{kinetic optimal transport} discrepancy $\sfd_T$ of~\cite{brigati2025kinetic} in order to characterise~\eqref{eq:nlkpm.} as the flow of maximal dissipation of the free energy for $m\ge1$. The restriction $m\le3/2$ appears, for instance, also in
~\cite{erbar2025gradientflowclassdiffusion}. Apart from~\cite{MR3478291}, which is devoted to a model of granular media with a different approach, our gradient flow interpretation is new in kinetic theory, even in the linear case $m=1$, while the (weak) convergence of~\eqref{eq:JKO.} is already known from~\cite{MR3271101} if $m=1$. The first key ingredient is the \emph{chain rule}, introduced in~\cite[Chapter~10,~pp.~228 \& 233 (E)]{MR2401600} in the context of Wasserstein gradient flows. The proof via regularisation by convolution originates from~\cite{MR4780487,MR4746872} using the \emph{Galilean group} as in~\cite{MR4527757}. The second key ingredient is strong compactness for~\eqref{eq:JKO.}, as discussed above. Similar considerations, adapted to a splitting scheme, are also crucial in~\cite{brigati2026fundamental}.

\medskip This paper is organised as follows. Theorem~\ref{thm1.} is proved in Section~\ref{Sec:ChainRule-GradienFlow}. The proof of Theorem~\ref{thm2.} is split into two sections: convergence of the JKO scheme to a solution to~\eqref{eq:nlkpm.} in Section~\ref{Sec:JKO}; maximal dissipation in the range~$m\in [1, 3/2]$ in Section~\ref{sec:sec4}.

\section{Chain rule for the free energy and gradient flow}
\label{Sec:ChainRule-GradienFlow}

This section is devoted to the proof of Theorem~\ref{thm1.}. Without further notice, we shall assume that all corresponding assumptions are satisfied.

\bigskip\noindent\textbf{Step 1: Regularisation via Galilean convolution.} Let $(\mu,F)=(\mu_t,F_t)_t$ be as in Theorem~\ref{thm1.} and denote by $f_t$ the density of $\mu_t$. We extend (in time) the force field $F$ by $0$ outside $[0,T]$ and then extend $f$ by solving \eqref{eq:vlasov.} for all times. Let $\psi$ be a smooth probability density supported on $[-1,0]$ with $\psi>0$ on $(-1,0)$ and for $\epsilon\in (0,1)$, let us consider the mollifier on~$\R\times \Gamma$ defined by
\[
\eta_\epsilon(t,x,v):=\frac1{\pi^d \, \epsilon^{1+2d}}\,\psi \Big(\frac{t}{\epsilon}\Big) \; \mathrm{e}^{-\,\frac{1}{\epsilon^2}(|x|^2+|v|^2)}.
\]
Inspired by~\cite{MR4527757}, we regularise $f$ and $F$ by
\begin{equation}\label{Def:Galilean}
f^\epsilon:=\eta_\epsilon\star f\quad\text{and}\quad F^\epsilon:=\big(\eta_\epsilon\star(F\,f)\big)/f^\epsilon
\end{equation}
where the \emph{Galilean convolution} $\star$ is defined by
\[
(\eta_\epsilon\star f)(t,x,v):=\int_\R\iint_\Gamma\eta_\epsilon\big(t-s,x-y-s\,(v-w),v-w\big)\,f(s,y,w)\,dy\,dw\,ds\,.
\]
According to~\cite[Lemma 4.9]{brigati2025kinetic}, $(f^\epsilon_t,F^\epsilon_t)_t$ then is a smooth classical solution to~\eqref{eq:vlasov.} on $\R\times \Gamma$. Moreover, it is a strong approximation in the sense that $f^\epsilon$ and $f^\epsilon F^{\epsilon}$ converge strongly in $\mathrm{L}^1_{\mathrm{loc}}(\R; \mathrm{L}^1(\Gamma))$ to $f$ and $fF$ respectively. Passing to a vanishing sequence of $\epsilon$'s, we may also assume that the previous convergence holds a.e.~in $(t,x,v)$ and, for a.e.~$t$, in $\mathrm{L}^1(\Gamma)$. Again following~\cite[Lemma 4.9]{brigati2025kinetic}, we have
\begin{equation}\label{eq:dpiF}
\int_0^T \|F^\epsilon_t\|^2_{\mathrm L^2(f_t^\epsilon)} \, dt \leq \int_0^T \|F_t\|^2_{\mathrm{L}^2(\mu_t)} \, dt
\end{equation}
for any $\epsilon>0$ and
\begin{equation}\label{eq:epsbound}
\begin{aligned}
& \lim_{\epsilon\to0_+}\int_0^T \iint_\Gamma |x|^2 \, f^\epsilon_t \, dt = \int_0^T \iint_\Gamma |x|^2 \, d\mu_t \, dt\,,\\
& \lim_{\epsilon\to0_+}\int_0^T \iint_\Gamma |v|^2 \, f^\epsilon_t \, dt = \int_0^T \iint_\Gamma |v|^2 \, d\mu_t \, dt\,.
\end{aligned}
\end{equation}
By \eqref{eq:dpiF}, $\sqrt{f^\epsilon} F^{\epsilon}$ is bounded in $\mathrm{L}^{2}\big((0,T)\times \Gamma\big)$ and, up to the extraction of sequences, converges weakly to some limit $G$. By strong $\mathrm L^2$ convergence of $\sqrt{f^\epsilon}$ to $\sqrt{f}$ and strong $\mathrm L^1$ convergence of $f^\epsilon F^{\epsilon}$ to $f F$ we readily have $f F= \sqrt{f} G$, , $G=\sqrt{f} F$. Weak $\mathrm{L}^2$ convergence of $\sqrt{f^\epsilon} F^{\epsilon}$ to $\sqrt{f} F$ is in fact strong because of the convergence of the $\mathrm{L^2}$ norm deduced from \eqref{eq:dpiF}. Hence, we have
\begin{equation}\label{eq:epsconv}
\lim_{\epsilon\to0_+}\left\|\,\sqrt{f^\epsilon} \, F^\epsilon-\sqrt{f} \, F\,\right\|_{\,\mathrm{L}^2((0,T)\times\Gamma)}=0\,.
\end{equation}
By the same reasoning, from the second condition in \eqref{eq:epsbound}, we deduce, after passing to a subsequence if necessary,
\begin{equation}\label{eq:cvmomentveps}
\lim_{\epsilon\to0_+}\big\|\,v\,\big(\sqrt{f^\epsilon}-\sqrt{f}\big)\,\big\|_{\,\mathrm{L}^2((0,T)\times\Gamma)}=0 \quad\text{and}\quad \lim_{\epsilon\to0_+} \iint_\Gamma |v|^2 \, f^\epsilon_t = \iint_\Gamma |v|^2 \, d\mu_t
\end{equation}
where the second equality holds for a.e.~$t$.

\bigskip\noindent\textbf{Step 2: Bounds in free energy and Fisher information.} We rely on Jensen's inequality and on the \emph{data processing inequality}: integrated convex functions decrease under convolution. See for instance~\cite[Chapter~2]{MR2239987}. If $m>1$, using the data processing inequality and letting $\epsilon\to0_+$ gives
\[
\Vert f_t\Vert^m_{\mathrm L^m(\Gamma)}\le\liminf_{\epsilon\to0_+}\Vert f^\epsilon_t \Vert^m_{\mathrm L^m(\Gamma)} \leq \limsup_{\epsilon\to0_+}\frac{1}{\epsilon} \int_{\R} \psi \Big(\frac{s}{\epsilon}\Big)\,\Vert f_{t-s} \Vert^m_{\mathrm L^m(\Gamma)}\,ds=\Vert f_t\Vert^m_{\mathrm L^m(\Gamma)}
\]
for a.e.~$t$, where the first inequality holds by Fatou's Lemma. For $m=1$, we can argue similarly, using second moments (see for instance~\cite{MR1617171}) to bound the negative part of $f^\epsilon \log (f^\epsilon)$, and Fatou's Lemma. Using again \eqref{eq:cvmomentveps}, for any $m\ge1$, we obtain
\begin{equation}\label{eq:dpie}
\lim_{\epsilon\to0_+}\Fm(\mu^\epsilon_t) =\Fm(\mu_t)\quad\text{for a.e. } t \in [0,T]\,.
\end{equation}

Next, let us turn our attention to the Fisher information $\mathcal{I}_m(f_t)$ which is more delicate. First of all, to define $\mathcal{I}_m(f_t)$ rigorously via \eqref{eq:redefinedFisher}, we shall use the following lemma (whose standard proof by approximation is left to the reader).
\begin{lemma} \label{lemma:chain}
Let~$\alpha\in(0,+\infty)$ and let~$L$ be a homogeneous first-order differential operator. If~$g \in \mathrm{L}^1_\mathrm{loc}(\Gamma)$ is nonnegative,~$L\,g \in \mathrm{L}^1_\mathrm{loc}(\Gamma)$, and~$g^{\alpha-1}\, L\,g \in \mathrm{L}^1_\mathrm{loc}(\Gamma)$, then
\[
L\,(g^\alpha) = \alpha \, g^{\alpha-1} \, L\,g \,.
\]
\end{lemma}
Applied with $L=\nabla_v$, $f=g^\alpha$, $\alpha=1/(m-1/2)$ and $m \in [1,3/2]$, if~$f \in \mathrm{L}^m_\mathrm{loc}(\Gamma)$, and~$\nabla_v (f^{m-1/2}) \in \mathrm{L}^2_\mathrm{loc}(\Gamma)$, then
\[
\nabla_v f = \tfrac{1}{m-1/2}\, f^{3/2-m}\, \nabla_v\(f^{m-1/2}\) \in \mathrm{L}^1_\mathrm{loc}(\Gamma) \,.
\]
The key observation now is that the function
\begin{equation}\label{Theta}
\R\times\R^d\ni(s,X) \mapsto \Theta(s,X) \coloneqq \begin{cases} s^{2\,m-3}\,|X|^2 \mbox{ if $s>0$\,,}\\ 0 \mbox{ if $s=0$ and $X=0$\,,}\\+\infty \mbox{ otherwise}\,,\end{cases}
\end{equation}
is jointly convex and lower semicontinuous for~$1\le m \le 3/2$ (and actually only in this range for $m$). Our assumptions on the curve $f$ then ensure that $\nabla_v f^{m-1/2}$ is $\mathrm{L^2}$, and thanks to Lemma \ref{lemma:chain}, the identity
\[
\frac{m^2}{(m-1/2)^2} \iint_\Gamma \big\vert \nabla_v f^{m-1/2} \big\vert^2 =m^2 \iint_\Gamma \Theta(f, \nabla_v f)
\]
holds for $f=f_t$ and a.e.~$t$. Recalling that we interpret~$\sqrt{f}\,\nabla_v P_m(f)$ as $\frac{m}{m-1/2}\,\nabla_v f^{m-1/2}$, we write
\[
m^2 \iint_\Gamma \Theta(f, \nabla_v f) = \big\Vert \sqrt{f}\,\nabla_v P_m(f)\big\Vert^2_{\mathrm{L}^2(\Gamma)} \, .
\]
These considerations lead to the following.
\begin{lemma}\label{Lem:Entropy-Fisher} For any $a$ and $b$ such that~$0<a<b<T$, we have
\begin{equation}\label{eq:dpiip}
\lim_{\epsilon\to0_+}\left\|\,\big(\nabla_v P_m(f^\epsilon) + v\big) \, \sqrt{f^\epsilon}-\big(\nabla_v P_m(f) + v\big) \, \sqrt f\,\right\|_{\,\mathrm{L}^2\big((a,b)\times\Gamma\big)}=0\,.
\end{equation}
In addition, for any fixed $\epsilon>0$, there is a constant $\mathcal C(T,\epsilon)>0$ such that
\begin{equation}\label{eq:dpiix}
\int_a^b \iint_\Gamma |\nabla_x P_m(f^\epsilon)|^2\,f^\epsilon\, dx\,dv\,dt \leq\mathcal C(T,\epsilon)\,.
\end{equation}
\end{lemma}
\begin{proof}
Thanks to \eqref{eq:cvmomentveps}, to establish \eqref{eq:dpiip}, we just need to prove that $\sqrt{f^\epsilon}\, \nabla_v P_m(f^\epsilon)$ converges to $\sqrt{f}\, \nabla_v P_m(f)$ in $\mathrm{L}^2\big((a,b)\times\Gamma\big)$. By~\eqref{Def:Galilean} and Lemma~\ref{lemma:chain}, we have
\[
\nabla_v f^\epsilon(t,x,v) = \int_{t}^{t+\epsilon} \iint_\Gamma \eta^\epsilon\big(t-s,x-y-s\,(v-w),v-w\big)\,\nabla_w f(s,y,w) \, dy\,dw\,ds\,.
\]
We now apply Jensen's inequality to the function $\Theta$ defined in~\eqref{Theta} and deduce from the assumption~$\int_0^T \mathcal{I}_m(f)\,dt < \infty$ that
\begin{equation} \label{eq:boundpr}
\begin{aligned}
&\int_a^b \iint_\Gamma |\nabla_v P_m(f^\epsilon)|^2\,f^\epsilon \, dx \, dv \, dt= m^2 \int_a^b \iint_\Gamma\Theta(f^\epsilon,\nabla_v f^\epsilon) \, dx \, dv \, dt\\
&\le m^2 \int_{a}^{b+\epsilon} \iint_\Gamma\Theta(f,\nabla_v f) \, dx \, dv \, dt = \int_a^b \iint_\Gamma |\nabla_v P_m(f)|^2\,f \, dx \, dv \, dt + o(1)\,.
\end{aligned}
\end{equation}
This implies that $\sqrt{f^\epsilon}\, \nabla_v P_m(f^\epsilon)$ is bounded in $\mathrm{L^2}$ and since $(f^\epsilon)^{m-1/2}$ converges to $f^{m-1/2}$ in $\mathrm{L}^{1/(m-1/2)}$, the weak limit of $\sqrt{f^\epsilon}\, \nabla_v P_m(f^\epsilon)$ necessarily is $\sqrt{f}\, \nabla_v P_m(f)=\frac{m}{m-1/2}\, \nabla_v f^{m-1/2}$. But thanks to~\eqref{eq:boundpr}, we see that the convergence is in fact strong. This completes the proof of~\eqref{eq:dpiip}.

\medskip To show \eqref{eq:dpiix}, we proceed in a similar way, except that the derivative with respect to~$x$ acts on the mollifier. More precisely, according to~\cite[p.~6]{MR4527757}, we have
\[
\nabla_x f^\epsilon(t,x,v) = \int_{t}^{t+\epsilon} \iint_\Gamma (\nabla_{y} \eta^\epsilon)(s,y,w) \, f\big(t-s,x-y-s\,(v-w),v-w\big) \, dy\,dw\,ds\,.
\]
Using the homogeneity of $\Theta$ and Jensen's inequality again yields
\[
\int_a^b\kern-4pt\iint_\Gamma\Theta(f^\epsilon,\nabla_x f^\epsilon)\,dx\,dv\,dt \leq (T+\epsilon)^{2\,m-1}\kern-4pt\int_{-\epsilon}^0 \iint_\Gamma \Theta(\eta^\epsilon,\nabla_x\eta^\epsilon) \, dx\,dv\,dt=:\mathcal C(T,\epsilon)
\]
which is finite for fixed $\epsilon>0$. This proves~\eqref{eq:dpiix}.\end{proof}

\bigskip\noindent\textbf{Step 3: Chain rule for regular solutions.}
Thanks to~\eqref{eq:boundpr}-\eqref{eq:dpiix}, the Wasserstein slope (see~\cite[Theorem~10.4.13]{MR2401600}) of $\Fm$ satisfies
\[
\iint_\Gamma \big|\nabla_{x,v} \, P_m(f^\epsilon) - (0,v)\big|^2 \, f^\epsilon \, dx\,dv = |\partial \Fm(f^\epsilon)|^2 \in \mathrm{L}^1(0,T)
\]
and $\Fm$ is convex along Wasserstein geodesics (see~\cite[Section 9.3]{MR2401600}). At the same time, $(\mu^\epsilon)_t$ is a $2-$absolutely continuous curve in the Wasserstein distance. According to~\cite[p.~233]{MR2401600}, the \emph{chain rule} applies for a.e.~$t \in (0,T)$, that is,
\begin{equation}\label{eq:cheps}
\frac{d}{dt} \, \Fm(\mu_t^\epsilon) = \iint_\Gamma \big( \nabla_x P_m(f^\epsilon) \cdot v \, f^\epsilon + \big(\nabla_v P_m(f^\epsilon) + v\big) \cdot \, F^\epsilon \, f^\epsilon \big) \, dx\,dv\,.
\end{equation}
The first term in the right-hand side, however, vanishes thanks to Stoke's theorem. Indeed, the function $\nabla_x P_m(f^\epsilon) \cdot v \, f^\epsilon$ is integrable on $\Gamma$ for a.e.~$t\in (a,b)$. Then, if $(\varphi_n)_{n\in\N}$ is a sequence of smooth and compactly supported cut-off functions such that $\varphi_n=1$ on centered balls of radius $n$, we obtain
\[
\begin{aligned}
\iint_\Gamma \nabla_x P_m(f^\epsilon) \cdot v \, f^\epsilon \, dx\,dv &= \lim_{n\to\infty} \iint_\Gamma \varphi_n \, \nabla_x P_m(f^\epsilon) \cdot v \, f^\epsilon \, dx\,dv\\
&= - \lim_{n\to\infty} \iint_\Gamma \nabla_x \varphi_n \cdot v \, \left(f^\epsilon \right)^m \, dx\, dv = 0
\end{aligned}
\]
because the field $v \, \left(f^\epsilon \right)^m$ belongs to $\mathrm{L}^1(\Gamma)$ for a.e.~$t\in (a,b)$. Taking this observation into account, we integrate~\eqref{eq:cheps} over the interval $(a,b)$ to get
\begin{equation}
\label{eq:chainproof2.}
\Fm(\mu^\epsilon_b) - \Fm(\mu^\epsilon_a) = \int_a^b \iint_\Gamma \big(\nabla_v P_m(f^\epsilon) +v \big) \, \cdot \, F^\epsilon \, f^\epsilon \, dx\,dv \, dt\,,
\end{equation}
which is~\eqref{eq:chainp.} for the regularised pair $(f^{\epsilon},F^{\epsilon})$.

\bigskip\noindent\textbf{Step 4: Conclusion.}
The equality in~\eqref{eq:chainp.} is proved by passing to the limit $\epsilon\to 0_+$ in~\eqref{eq:chainproof2.}, thanks to~\eqref{eq:epsconv}-\eqref{eq:dpie}-\eqref{eq:dpiip}. Since~$(\nabla_v P_m(\mu_t)+v)$ is a $v$-gradient (precisely, in the~$\mathrm{L}^2(\mu)$-closure of gradients of smooth functions), we may write
\[
\Fm(\mu_b) - \Fm(\mu_a) = \int_a^b \iint_\Gamma \big(\nabla_v P_m(\mu_t) + v\big) \cdot (F_t +G_t) \, \mu_t \, dx\,dv \, dt \,,
\]
for any $G_t \in \mathrm{L}^2(\mu_t;\mathbb R^d)$ with $\nabla_v \cdot (G_t \, \mu_t) = 0$ for a.e.~$t \in (a,b)$. We apply Young's inequality
\[
\begin{aligned}
\int_a^b \iint_\Gamma \big(\nabla_v P_m(\mu_t) + v\big) \,\cdot\,& (F_t +G_t) \, \mu_t \, dx\,dv \, dt \, \\
\geq - \frac{1}{2} \, \int_a^b &\left( \,\|F_t + G_t \|^2_{\mathrm{L}^2(\mu_t)} + \mathcal{I}_m(\mu_t) \, \right) \, dt\,,
\end{aligned}
\]
take the infimum in the right-hand side over all $(G_t)_t$, and obtain the norm of the projection of $F_t$ over $v$-gradients, that is, $|\mu'_t|_\sfd$, by~\cite[Proposition~5.23]{brigati2025kinetic}. This proves~\eqref{eq:chainp.}.

Equality is achieved if and only if the two fields $F_t$ and $\nabla_v P_m(f_t)+v$ are opposite, up to adding a force field $(G_t)_t$ such that $\nabla_v \cdot (G_t \, \mu_t) = 0$. In this case,~\eqref{eq:chainp.} holds with
\[
|\mu'_t|_\sfd = \lVert \nabla_v P_m(f_t)+v \rVert_{\mathrm{L}^2(\mu_t)} = \mathcal{I}_m(\mu_t) \, .
\]

\section{Strong convergence of the minimising-movement scheme}
\label{Sec:JKO}

All computations are performed for $m>1$, but, with minor changes, can be extended to the linear case $m=1$.

\bigskip\noindent\textbf{Step 1: Basic estimates for a JKO type scheme.}
Let~$h \in (0,1)$. On the phase space~$\Gamma$, the transformation
\[
\Gh(x, v)\coloneqq(x+h\,v, v) \, , \qquad (x, v)\in \Gamma
\]
preserves Lebesgue's measure as well as the velocity marginal. As a consequence we have $\Fm(\mu)=\Fm({\Gh}{}_\# \mu)$ for every $\mu\in \Pg$, where ${\Gh}{}_\# \mu$ denotes the push forward of $\mu$ through $\Gh$. Following~\cite[Section~3.1]{brigati2025kinetic}, one can rewrite $\sfd_h^2$ defined by~\eqref{eq:d_T.} as the value of the optimal transport problem
\[
\sfd_h^2 (\mu, \nu)=\inf_{\gamma \in \Pi(\mu, \nu)} \iint_{\Gamma\times \Gamma} \ch\big((x,v), (y,w)\big)\, d \gamma \big((x,v), (y,w)\big)
\]
where $\Pi(\mu, \nu)$ is the set of probability measures on $\Gamma\times \Gamma$ with marginals $\mu$ and $\nu$, and 
\begin{equation}\label{ch}
\ch\big((x,v), (y,w)\big)\coloneqq 12 \, \left\vert \frac{y-x}{h}- \frac{v+w}{2} \right\vert^2 + \vert w-v\vert^2= \big\vert(y,w)- \Gh(x,v)\big\vert_h^2\,.
\end{equation}
Here, the squared norm $\vert \cdot \vert^2_h$ on the phase space is defined by
\[
\big\vert (a, b) \big\vert_h^2 \coloneqq 12 \, \left \vert \frac{a}{h}- \frac{b}{2} \right\vert^2 + \vert b\vert^2= \big\vert A_h(a, b) \big\vert^2 \, , \quad\mbox{with } A_h(a,b) \coloneqq \(\frac{2\,\sqrt{3}}{h}\,a-\sqrt{3}\,b,b\) \, .
\]
The cost function $\ch$ appears for instance in~\cite[(4.1)]{MR2560310}. With this notation, we have
\begin{equation}\label{eq:formulasforch}
\sfd_h^2 (\mu, \nu)=W_h^2\big({\Gh}{}_\# \mu, \nu\big)= W^2\big({(A_h\circ \Gh)}{}_\#\mu, {A_h}{}_\#\nu\big)
\end{equation}
where $W_h$ is the Wasserstein distance associated with the norm $\vert \cdot \vert_h$ and $W$ is the usual $2$-Wasserstein distance. See~\cite{MR2560310} for details. Notice that the scheme~\eqref{eq:JKO.} may also be written as
\begin{equation}\label{eq:JKOrewritten}
\mu^{(h)}_{k+1} \in \argmin_{\nu \in \Pg} \left\{ \Fm(\nu) + \frac{1}{2\,h} \, W_h^2 \({\Gh}{}_\# \mu^{(h)}_k, \nu\)\right\}, \quad k \in \N \,, \quad \mu_0^{(h)}=\mu_0\,.
\end{equation}
Standard lower semicontinuity and convexity arguments ensure that this scheme is well-posed, and the explicit form of the cost $\ch$ guarantees the uniqueness of an optimal transport plan $\gamma_k^{(h)} \in \Pi\big(\mu_k^{(h)}, \mu_{k+1}^{(h)}\big)$ for the cost $\ch$. We define the \emph{piecewise constant in time map} $t \mapsto \mu^{(h)}_t$ by
\[
\mu^{(h)}_t\coloneqq\mu_{k+1}^{(h)}\,, \quad t\in \big(k\,h, (k+1)\,h\big]\,, \quad k \in \N \,, \quad \mu^{(h)}(0)=\mu_0\,.
\]

Taking ${\Gh}{}_\# \mu^{(h)}_k$ as a competitor in the minimisation~\eqref{eq:JKOrewritten}, we have
\begin{equation}\label{eq:edi0}
\Fm\(\mu_{k+1}^{(h)}\) +\frac{1}{2\,h} \, W_h^2\({\Gh}{}_\# \mu_k^{(h)}, \mu_{k+1}^{(h)}\) \leq \Fm\( {\Gh}{}_\# \mu_k^ {(h)}\)= \Fm\(\mu_k^{(h)}\) \, ,
\end{equation}
and, in particular,
\begin{equation}
 \label{eq:unif_boundLm}
 \max_{k \in \N } \Fm\(\mu_k^{(h)}\) \le \Fm\(\mu_0\) \, .
\end{equation}

Let us fix a time horizon~$T > 0$ and define~$N=\lceil {T}/{h} \rceil$. Summing~\eqref{eq:edi0} over~$k$ and using the fact that $\Fm$ is nonnegative yield
\begin{equation}\label{eq:basicjko1}
\sum_{k=0}^{N-1} \frac{1}{2\,h} \, W_h^2\({\Gh}{}_\# \mu_k^{(h)}, \mu_{k+1}^{(h)}\)= \frac{1}{2\,h} \sum_{k=0}^{N-1} \sfd_h^2\(\mu_k^{(h)}, \mu_{k+1}^{(h)}\) \leq \Fm(\mu_0) \,.
\end{equation}
Since moments in velocity are bounded by~\eqref{eq:unif_boundLm}, we have
\[
W^2\(\mu_{k}^{(h)}, {\Gh}{}_\# \mu_k^{(h)}\) \le \Cst\,h^2\,, \qquad k\in \N\,.
\]
Together with the crude inequality~$W^2 \leq 2 \, W_h^2$ and~\eqref{eq:basicjko1}, we obtain
\begin{equation}\label{eq:basicjko4}
\sum_{k=0}^{N-1} \frac{1}{2\,h} \, W^2\(\mu_k^{(h)}, \mu_{k+1}^{(h)}\)\le\Cst\,\frac{1}{h} \sum_{k=0}^{N-1} \(h^2+ W_h^2\({\Gh}{}_\# \mu_k^{(h)}, \mu_{k+1}^{(h)}\)\)\le\Cst\,.
\end{equation}
for some constant $\Cst$ depending on $T$. In what follows, the constant $\Cst$ may change from formula to formula but stays independent of $h$. As a consequence,
\begin{equation}\label{eq:muhholder}
W\(\mu^{(h)}_t, \mu^{(h)}_s\) \le\Cst\,\sqrt{\,\vert t-s\vert + h}\,, \qquad 0\le s,\,t \le T \,.
\end{equation}

Finally, using the triangle inequality, Cauchy--Schwarz, and~\eqref{eq:basicjko4} we get
\begin{multline}\label{eq:basicjko2}
\max_{k \in \{0,\dots,N\} } \iint_\Gamma \vert (x,v)\vert^2 \, d \mu_k^{(h)} \le 2\iint_\Gamma \vert (x,v) \vert^2\,d \mu_0+ 2\,\, \max_{k \in \{0,\dots,N\}} W^2\(\mu_0, \mu_k^{(h)}\) \\
 \leq \Cst + 2 \(\sum_{k=0}^{N-1} W \( \mu_k^{(h)} , \mu_{k+1}^{(h)}\) \)^2 \leq \Cst + 2\,N \sum_{k=0}^{N-1} W^2 \( \mu_k^{(h)} , \mu_{k+1}^{(h)}\) \le\Cst \, . 
\end{multline}

\bigskip\noindent\textbf{Step 2: Euler-Lagrange equations and approximate nonlinear kinetic Fokker--Planck equation.} Let~$S_k^{(h)}=\Big(\left[S_k^{(h)} \right]_1, \left[S_k^{(h)} \right]_2\Big)$ be the~$\mathsf d_h$-optimal transport map from~$\mu_{k+1}^{(h)}$ to~$\mu_k^{(h)}$. For every smooth vector field~$V \colon \Gamma \to \Gamma$, the first variation
\[
\frac{d^+}{d\epsilon} \left[\Fm\left((\mathrm{id}+\epsilon\,V)_\# \mu_{k+1}^{(h)}\right) + \frac1{2\,h}\iint_{\Gamma \times \Gamma} \ch\,d\(S_k^{(h)}, \mathrm{id}+\epsilon\,V\){}_\# \mu_{k+1}^{(h)}\right]
\]
(with $\ch$ given by~\eqref{ch}) at $\epsilon = 0$ is nonnegative. Therefore, tedious but standard computations (see, e.g.,~\cite{MR1617171,MR3584930}) yield~$\mu_{k+1}^{(h)} = f \, dx \, dv$ with~$f^m \in W^{1,1}(\Gamma)$ and
\begin{equation}\label{eq:formulaforSh}
\begin{aligned}
\left[S_k^{(h)}(y,w)\right]_1&=y-h\,w-\frac{h^2}{2} \left(\frac {\nabla_v (f^m)}{f}+w \right) - \frac{h^3}{6} \frac{\nabla_x (f^m)}{f}\,,\\
\left[S_k^{(h)}(y,w)\right]_2&=w+h\left( \frac{D_h (f^m)}{f}+w \right)\,,
\end{aligned}
\end{equation}
with $D_h \coloneqq \frac{h}{2}\, \nabla_x + \nabla_v$. By plugging this in the cost, we obtain
\begin{align} \label{eq:sfdDP}
\sfd^2_h\(\mu_k^{(h)}, \mu_{k+1}^{(h)}\) &= \iint_{\Gamma \times \Gamma} \ch \, d \big(S_k^{(h)}, \mathrm{id}\big){}_\# \mu_{k+1}^{(h)} \nonumber\\
&= \frac{h^4}{12} \iint_\Gamma \frac{|\nabla_x (f^m)|^2}{f} \, dy \, dw + h^2 \iint_\Gamma \left|\, \frac{D_h(f^m)}{{f}} + w \,\right|^2 \, d\mu_{k+1}^{(h)} \\
&\ge \frac{h^2}{4} \iint_\Gamma \left|\, \frac{\nabla_v(f^m)}{{f}} + w \,\right|^2 \, d\mu_{k+1}^{(h)} \, , \nonumber
\end{align}
where the inequality is obtained from~$\frac{a^2}{4} \le \frac{b^2}{3} + (a+b)^2$. By Lemma~\ref{lemma:chain} applied with~$\alpha=1-1/(2\,m)$ and~$g=f^m$, we find $f^{m-\frac{1}{2}} \in \mathrm{W}^{1,1}_\mathrm{loc}(\Gamma)$ and we can write
\begin{equation} \label{eq:chain_m_mhalf}
\frac{1}{f}\,L (f^m) = \frac{1}{\sqrt{f}}\,\frac{m}{m-1/2}\,L \left(f^{m-1/2} \right) \, , \quad \text{for both } L = \nabla_v \text{ and } L = D_h \, .
\end{equation}
In particular,
\begin{multline}
\sfd^2_h\(\mu_k^{(h)}, \mu_{k+1}^{(h)}\) \ge \frac{h^2}{4}\iint_\Gamma\(\big(\tfrac{2\,m}{2\,m-1}\big)^2\,\big|\nabla_v\big(f^{m-1/2}\big)\big|^2\,+|v|^2\,f-2 \, v \cdot \nabla_v (f^m)\) \, dx\,dv \\
= \frac{h^2}{4}\iint_\Gamma\(\big(\tfrac{2\,m}{2\,m-1}\big)^2\,\big|\nabla_v\big(f^{m-1/2}\big)\big|^2\,+|v|^2\,f-2\,d\,f^m \) \, dx\,dv = \IFm\(\mu_{k+1}^{(h)}\) \label{eq:sfdD-2} \, ,
\end{multline}
where integration by parts can be proven using~$f \in \mathrm{L}^m(\Gamma)$. Setting
\begin{equation}\label{eq:defdeFhQh}
F^{(h)}_t \coloneqq-\left(\frac{D_h (f^m)}{f}+v\right) = -\left(\frac{m}{m-1/2} \frac{D_h (f^{m-1/2})}{\sqrt{f}}+v\right) \,, \quad d\mu_{k+1}^{(h)} = f \, dx \, dv
\end{equation}
for any $t \in \big(k\,h,(k+1)\,h\big)$, Equation~\eqref{eq:sfdDP} gives
\begin{equation}\label{eq:metricfisher}
\frac{1}{2\,h} \, \sfd^2_h\(\mu_k^{(h)}, \mu_{k+1}^{(h)}\) \geq \frac{1}{2} \int_{k\,h}^{(k+1)\,h} \big\Vert F^{(h)}_t \big\Vert^2_{\mathrm L^2\(\mu^{(h)}_t\)}\, dt.
\end{equation}
Combining~\eqref{eq:sfdD-2},~\eqref{eq:metricfisher}, and~\eqref{eq:basicjko1}, we deduce
\begin{equation} \label{eq:boundF}
	\int_0^T \big\Vert F_t^{(h)} \big\Vert_{\mathrm L^2\(\mu^{(h)}_t\)}^2 \, dt + \int_0^T \mathcal{I}_m\big(\mu^{(h)}_t\big) \, dt \le\Cst\,\frac{1}{h} \sum_{n=0}^{N-1} \sfd^2_h\(\mu_k^{(h)}, \mu_{k+1}^{(h)}\) \le\Cst\,,
\end{equation}
Let $\eta \in \mathcal{D}\big((0,T)\times \Gamma\big)$. Using the fact that $\mu_k^{(h)}= \left(S_k^{(h)}\right)_\# \mu_{k+1}^{(h)}$, we have
\begin{align}
\int_0^T\iint_\Gamma \partial_t \eta \, d\mu^{(h)}_t\,dt&=\sum_{k=1}^{N-1} \iint_\Gamma \eta(k\,h,\cdot) \, d\(\mu_k^{(h)}-\mu_{k+1}^{(h)}\)\nonumber\\
&=\sum_{k=1}^{N-1} \iint_\Gamma \left(\eta\(k\,h, S_k^{(h)}\)-\eta(k\,h, \cdot)\right) \, d\mu_{k+1}^{(h)} \nonumber\\
& =\sum_{k=1}^{N-1} \iint_\Gamma \left(\nabla_{x,v} \eta(k\,h,\cdot) \cdot \(S_k^{(h)}-\mathrm{id}\)+ R_k^{(h)}\right) \, d \mu_{k+1}^{(h)} \, , \label{eq:ippdh}
\end{align}
where the error term $R_k^{(h)}$ in the second-order Taylor expansion satisfies
\[
\big\vert R_k^{(h)} \big\vert \leq \frac{1}{2}\,\Vert D^2 \eta \Vert_{\mathrm L^\infty} \left\vert\, S_k^{(h)}-\mathrm{id} \,\right\vert^2\,.
\]
Writing~$\eta_k \coloneqq \eta(k\,h, \cdot)$ and, once again,~$\mu_{k+1}^{(h)} = f \, dx \, dv$, we observe that, by~\eqref{eq:formulaforSh},
\begin{multline*}
-\iint_\Gamma \nabla_x \eta_k \cdot \(\big[S_k^{(h)}\big]_1-y\) \, d \mu_{k+1}^{(h)} = \left(h+\frac{h^2}{2}\right) \iint_\Gamma \nabla_x \eta_k \cdot w \, d \mu_{k+1}^{(h)}\\
- \frac{h^2}{2} \iint_\Gamma \mathrm{div}_v(\nabla_x \eta_k)\, f^m \, dx \, dv - \iint_\Gamma \frac{h^3}{6}\,\Delta_x \eta_k\,f^m \, dx \, dv \,,
\end{multline*}
whence, recalling~\eqref{eq:unif_boundLm}, and allowing~$\Cst$ to depend on~$\eta$,
\begin{equation} \label{eq:freetransport}
\left|\iint_\Gamma \nabla_x \eta_k \cdot \(\big[S_k^{(h)}\big]_1-y+h\,w\) \, d \mu_{k+1}^{(h)}\right| \le\Cst\,h^2\, \Fm\(\mu_{k+1}^{(h)}\) \le\Cst\,h^2 .
\end{equation}
Furthermore,
\begin{align*}
&\frac{1}{2}\,\big\Vert S_k^{(h)}-\mathrm{id} \big\Vert^2_{\mathrm L^2\(\mu_{k+1}^{(h)}\)} \leq \big\Vert S_k^{(h)}-\Gh \circ S_k^{(h)} \big\Vert^2_{\mathrm L^2\(\mu_{k+1}^{(h)}\)} + \big\Vert \Gh \circ S_k^{(h)}-\mathrm{id} \big\Vert^2_{\mathrm L^2\(\mu_{k+1}^{(h)}\)}\\
&\quad \leq \big\Vert \mathrm{id}-\Gh \big\Vert^2_{\mathrm L^2(\mu_{k}^{(h)})} + 2\iint_\Gamma \big\Vert \Gh \circ S_k^{(h)}-\mathrm{id}\big\Vert_h^2 \, d\mu_{k+1}^{(h)} \\
&\le h^2 \iint_\Gamma \big\Vert v\big\Vert^2\,d\mu_{k}^{(h)} + 2\, \sfd_h^2\(\mu_k^{(h)}, \mu_{k+1}^{(h)}\) \le\Cst\, \( h^2 + \sfd_h^2\(\mu_k^{(h)}, \mu_{k+1}^{(h)}\)\)
\end{align*}
where we used $\mu_k^{(h)}= {S_k^{(h)}}{}_\# \mu_{k+1}^{(h)}$, and $\vert \cdot \vert^2 \leq 2\,\vert \cdot \vert_h^2$ in the second line, and the optimality of $S_k^{(h)}$ and~\eqref{eq:basicjko2} in the third line. By~\eqref{eq:basicjko1}, we obtain that $\sum_{k=1}^{N-1} \iint_\Gamma \big\vert R_k^{(h)}\big\vert \, d \mu_{k+1}^{(h)} \le\Cst\,h$. Hence, using~\eqref{eq:ippdh},~\eqref{eq:freetransport}, and~\eqref{eq:formulaforSh}, we obtain
\begin{align*}
\int_0^T\iint_\Gamma\partial_t \eta \, d \mu^{(h)}_t\,dt
&= -\int_0^T\iint_\Gamma \nabla_x \eta \cdot v \, d\mu^{(h)}_t\,dt\\
&-\int_0^T\iint_\Gamma \nabla_v \eta \cdot F^{(h)} \, d \mu^{(h)}_t\,dt + O(h) \,.
\end{align*}
In other words, we have shown that, in the distributional sense,
\begin{equation}\label{eq:approxvlasov}
\partial_t \mu^{(h)} + v \cdot \nabla_x \mu^{(h)} + \nabla_v \cdot \(F^{(h)}\, \mu^{(h)}\) \to 0 \quad \mbox{as $h \to 0_+$}\,.
\end{equation}
In conclusion of this step, observe that the vector measure~$Q^{(h)} \coloneqq F^{(h)}\, \mu^{(h)}$ is bounded in total variation uniformly in~$h$, because, by the Cauchy--Schwarz inequality and~\eqref{eq:boundF},
\begin{equation}\label{eq:boundQhL1}
\Vert Q^{(h)}\Vert_{\mathrm {TV}} = \int_0^T \big\Vert F^{(h)}_t \big\Vert_{\mathrm L^1\(\mu^{(h)}_t\)} \, dt \leq \int_0^T \big\Vert F^{(h)}_t \big\Vert_{\mathrm L^2\(\mu^{(h)}_t\)} \, dt \le\Cst\,.
\end{equation}

\bigskip\noindent\textbf{Step 3: $\mathrm L^1$-contraction and $\mathrm L^1$-equicontinuity in $x$.} In~\cite{MR4553540}, an $\mathrm L^1$-contraction principle is established for JKO steps which, in particular, implies the following. Given $R>0$ and absolutely continuous probabilities $\mu_1$ and $\mu_2$ supported on $B_R$, and denoting by
\[\nu_i \coloneqq\argmin_{\vspace*{-4pt}\begin{array}{c}\scriptstyle\nu \in \Pg\\[-4pt]\scriptstyle\spt(\nu)\subset B_R\end{array}} \left\{\Fm(\nu) + \tfrac{1}{2\,h} \, W_h^2(\mu_i, \nu) \right\},\quad i=1,2 \,,\]
we have
\[\Vert (\nu_1-\nu_2)_+\Vert_{\mathrm L^1(B_R)} \leq \Vert (\mu_1-\mu_2 )_+\Vert_{\mathrm L^1(B_R)} \,.\]
Let us briefly explain how to to extend this principle to the full space by approximation and $\Gamma$-convergence arguments. Let $\mu_i$ be in $\Pg\cap \mathrm L^1(\Gamma)$, let $(R_n)_{n\in\N}$ be a sequence of radii such that $R_n \uparrow \infty$ and $\mu_i(B_{R_n}) \geq 1/2$, and let $\mu_i^n$ be the normalised restriction of~$\mu_i$ to $B_{R_n}$, \emph{i.e.}, $\mu_i^n(A)=\mu_i(A\cap B_{R_n})/\mu_i(B_{R_n})$. For~$\nu \in \Pg$, define the functionals
\[J_i^n(\nu):= \begin{cases} \Fm(\nu) + \frac{1}{2\,h} \, W_h^2(\mu_i^n, \nu) &\mbox{ if $\spt(\nu) \subset B_{R_n}$,}\\ + \infty &\mbox{ otherwise,}\end{cases}\]
and denote by $\nu_i^n$ the minimiser of $J_i^n$. It is easy to check that the sequence $(J_i^n)_{n\in\N}$ \hbox{$\Gamma$-converges} to $\Fm + \frac{1}{2\,h} \, W_h^2(\mu_i, \cdot)$ as $n\to \infty$ for the narrow topology (to construct recovery sequences, one can use the same approximation by restriction and normalisation as in the definition of $\mu_i^n$). This entails narrow (in fact, weak $\mathrm L^m$) convergence of $\nu_i^n$ to a minimiser~$\nu_i$ of $\Fm + \frac{1}{2\,h} \, W_h^2(\mu_i, \cdot)$. By the $\mathrm L^1$-contraction principle of \cite{MR4553540},
\[\Vert (\nu_1^n-\nu_2^n)_+\Vert_{\mathrm L^1(B_{R_n})} \leq \Vert (\mu_1^n-\mu_2^n )_+\Vert_{\mathrm L^1(B_{R_n})} \, ,\]
and, since $\mu_i^n$ converges to $\mu_i$ in $\mathrm L^1$ and using the weak lower semicontinuity properties of the left hand side, by passing to the liminf, we readily obtain the desired inequality
\[\Vert (\nu_1-\nu_2)_+\Vert_{\mathrm L^1(\Gamma)} \leq \Vert (\mu_1-\mu_2)_+\Vert_{\mathrm L^1(\Gamma)} \, .\]

\smallskip In our JKO scheme, recall that $\mu_{k+1}^{(h)}$ is obtained from $\mu_k^{(h)}$ through~\eqref{eq:JKOrewritten}. For~$\delta\in \R^d$, set $\tau_\delta(x,v) \coloneqq (x-\delta, v)$. The invariance of $\Fm$ by horizontal translations gives
\[ {\tau_\delta}{}_\# \,\mu_{k+1}^{(h)} = \argmin_{\nu \in \Pg} \left\{\Fm(\nu) + \frac{1}{2\,h} \, W_h^{2} \( (\tau_\delta \circ {\Gh})_\# \mu_k^{(h)}, \nu\)\right\} \,, \]
which, combined with the above $\mathrm L^1$-contraction principle and the commutation of $\Gh$ with horizontal translations, yields
\begin{equation*}
\big\Vert {\tau_\delta}{}_\#\,\mu_{k+1}^{(h)}- \mu_{k+1}^{(h)} \big\Vert_{\mathrm L^1(\Gamma)} \leq \big\Vert (\tau_\delta\circ{\Gh})_\# \mu_k^{(h)} - {\Gh}{}_\# \mu_k^{(h)} \big\Vert_{\mathrm L^1(\Gamma)} = \big\Vert {\tau_\delta}{}_\# \mu_k^{(h)} - \mu_k^{(h)} \big\Vert_{\mathrm L^1(\Gamma)} \,,
\end{equation*}
so that for all $k$ and all $\delta\in \R^d$, we have
\begin{equation*}
\big\Vert {\tau_\delta}{}_\# \,\mu_k^{(h)} - \mu_k^{(h)} \big\Vert_{\mathrm L^1(\Gamma)} \leq \Delta_{0}(\delta)\coloneqq\big\Vert {\tau_\delta}{}_\#\,\mu_0 - \mu_0 \big\Vert_{\mathrm L^1(\Gamma)} \,.
\end{equation*}
We have thus shown that
\begin{equation}\label{eq:L1equicontinxhatmu0}
\sup_{t \ge 0 }\big\Vert {\tau_\delta}{}_\#\,\mu^{(h)}_t - {\mu}^{(h)}_t \big\Vert_{\mathrm L^1(\Gamma)} \leq \Delta_{0}(\delta)\,.
\end{equation}
Since $\mu_0$ lies in $\mathrm L^1(\Gamma)$, we have $\lim_{\delta \to 0_+}\Delta_{0}(\delta)=0$.

\bigskip\noindent\textbf{Step 4: Strong~$\mathrm{L}^1$-compactness.}
We are going to show that ${\mu}^{(h)}$ is relatively compact in $\mathrm L^1\big((0,T)\times \Gamma\big)$. To this end, we will make use of the refined version of the Aubin-Lions compactness Theorem due to Rossi and Savar\'e~\cite{MR2005609}. Notice that, in addition to the contractivity~\eqref{eq:L1equicontinxhatmu0}, we have a uniform bound $C$ on the moments $\sup_{t \in [0,T]} \iint_\Gamma \vert (x, v)\vert^2 \,d{\mu}^{(h)}_t$ by~\eqref{eq:basicjko2}, as well as on~$\int_0^T \IFm\big({\mu}^{(h)}_t\big)\, dt$ thanks to~\eqref{eq:boundF}. Define $K$ as the set of functions~$f\in \mathrm L^1\cap \Pg$ that satisfy
\begin{equation} \label{eq:defK} \iint_\Gamma |(x,v)|^2 f \, dx \, dv \leq C \, , \quad \left\Vert f \circ \tau_\delta^{-1}- f \right\Vert_{\mathrm L^1(\Gamma)} \leq \Delta_{0}(\delta)\quad\forall\, \delta \in \R^d \,, \end{equation}
and set
\[\mathscr J_m(f)\coloneqq\begin{cases} \IFm(f) &\mbox{ if $f\in K$,}\\+\infty &\mbox{ if $f\in \mathrm{L}^1(\Gamma) \setminus K$.} \end{cases}\]
We have
\[\sup_h \int_0^T \mathscr J_m \(\mu^{(h)}(t)\) dt \leq C\]
and $W\({\mu}^{(h)}(t), {\mu}^{(h)}(s)\) \le\Cst\,\sqrt{\,\vert t-s\vert + h}$ for any $(t,s) \in [0,T]^2$ by~\eqref{eq:muhholder}. 
If we prove that sublevel sets of $\mathscr J_m $ are relatively compact in $\mathrm L^1(\Gamma)$,~\cite[Theorem~2]{MR2005609} will enable us to conclude. To this end, we will invoke the Riesz--Fr\'echet--Kolmogorov Theorem on a sublevel set of~$\mathscr J_m $. Indeed, it follows from the uniform moment bound that any such sublevel is equitight, and we already have equicontinuity for~$x$-translations thanks to the second property in~\eqref{eq:defK}. There only remains to prove equicontinuity for~$v$-translations. Fix a sublevel set~$S$ of~$\mathscr J_m $. There exists a constant~$C_S$ such that $\Vert \nabla_v f^{m-1/2}\Vert_{\mathrm L^2(\Gamma)}\leq C_S$ for all~$f \in S$. Fix one such~$f$. Denoting~$g \coloneqq f^{m-1/2}$, one can express $f=g^{2\theta}$ for the exponent $\theta=1/(2m-1)$, which is smaller than $1$ since~$m\geq 1$. For every~$\delta \in \mathbb R^d$ and~$R > 0$, denoting by~$B_R$ the ball of radius~$R$ centered at~$0$, it follows that
\begin{align*}
&\iint_{B_R \times B_R} \vert f(x, v+\delta)-f(x,v)\vert\, dv\, dx \\
&\quad = \iint_{B_R \times B_R} \vert g(x, v+\delta)^\theta-g(x, v)^\theta\vert\, ( g(x, v+\delta)^\theta+g(x, v)^\theta) \,dv\,dx\\
&\quad \le \iint_{B_R \times B_R} \vert g(x, v+\delta)-g(x, v)\vert^\theta \, ( g(x, v+\delta)^\theta+g(x, v)^\theta) \,dv\,dx\\
&\quad \le 2 \, |B_R|^{1-\theta} \, \( \iint_{B_R \times B_R} \vert g(x, v+\delta)-g(x, v)\vert^2 \, dv \, dx \)^{\theta/2} \, \(\iint_{\Gamma} g^{2\theta} \,dv\,dx \)^{1/2} \\
&\quad \le 2 \, |\delta|^\theta \, |B_R|^{1-\theta} \( \iint_{\Gamma} |\nabla_v (f^{m-1/2})|^2 \, dv \, dx \)^{\theta/2} \(\iint_{\Gamma} f \,dv\,dx \)^{1/2} \le C_R (C_S \, |\delta|)^{\theta} \, ,
\end{align*}
where we used the fact that $\theta \le 1$ in the third line, and H\"{o}lder's inequality with exponents~$2/(1-\theta)$,~$2/\theta$, and~$2$ in the third line. Therefore,
\[ \Vert f(\cdot, \delta + \cdot)- f \Vert_{\mathrm L^1(\Gamma)} \le C_R \, \big(C_S \, |\delta| \big)^{\theta} + C \, \left(R^{-2} + \big(R-|\delta|\big)^{-2} \right) \]
for~$R > |\delta|$, which yields the desired $v$-equicontinuity.

\bigskip\noindent\textbf{Step 5: Limit points solve~\eqref{eq:nlkpm.}.}
Let~$\mu$ be a~$\mathrm{L}^1_\mathrm{loc}\big(\R^+; \mathrm{L}^1(\Gamma)\big)$-limit point of~$\mu^{(h)}$. We observe the following. First, thanks to~\eqref{eq:muhholder} and the refined version of the Arzel\`a--Ascoli Theorem from~\cite[Proposition 3.3.1]{MR2401600}, we may assume that $\mu \in \mathrm{C}^{0,1/2}\big([0,T], (\Pg, W)\big)$ with
\begin{equation*}
\lim_{h \to 0_+}\sup_{t\in [0, T]} W\(\mu^{(h)}_t, \mu_t\)=0\,.
\end{equation*}
In particular,~$\mu(0) = \mu_0$. Secondly, by indentifying~$\mu^{(h)}$ and $\mu$ with their densities with respect to the Lebesgue measure, we have~$\sqrt{\mu^{(h)}} \to \sqrt{\mu}$ in~$\mathrm{L}^2\big((0,T) \times \Gamma\big)$. Thirdly, by the uniform bound~\eqref{eq:unif_boundLm} on the~$\mathrm{L}^m$-norm and the Dunford--Pettis Theorem, we may assume that~$\big(\mu^{(h)}\big)^{m-1/2}$ converges weakly to~$\mu^{m-1/2}$ in~$\mathrm{L}^1\big((0,T) \times \omega\big)$ for all bounded domain~$\omega \subset \Gamma$. Lastly, by~\eqref{eq:boundF}, we may also assume that~$F^{(h)}\sqrt{\mu^{(h)}}$ converges weakly in~$\mathrm{L}^2\big((0,T)\times \Gamma;\R^d\big)$ to some function~$G$.

{}From the previous observations and using~\eqref{eq:defdeFhQh}, we obtain, for any test function~$\eta \in \mathcal{D}\big((0,T) \times \Gamma\big)$,
\begin{align*}
 \langle \eta, G \rangle &= \lim_{h \to 0_+} \( \tfrac{m}{m-1/2} \, \big\langle D_h \eta, (\mu^{(h)})^{m-1/2} \big\rangle - \left\langle v \, \eta, \sqrt{\mu^{(h)}} \, \right\rangle \)\\
 &= \tfrac{m}{m-1/2} \, \big\langle \nabla_v \eta, \mu^{m-1/2} \big\rangle - \langle v \, \eta, \sqrt{\mu} \rangle \, ,
\end{align*}
hence
\[
 G = -\left(\,\tfrac{m}{m-1/2} \, \nabla_v \big(\mu^{m-1/2}\big) + v \, \sqrt{\mu} \,\right) \, .
\]
In the light of~\eqref{eq:approxvlasov}, and again by the three convergences observed above, for any test function~$\eta$, we find
\begin{align*}
0 &= \lim_{h \to 0_+} \( \big\langle (\partial_t + v \cdot \nabla_x)\,\eta, \mu^{(h)}\big\rangle + \left\langle \nabla_v \eta \, \sqrt{\mu^{(h)}}, F^{(h)}\, \sqrt{\mu^{(h)}} \, \right\rangle \) \\
&= \big\langle (\partial_t + v \cdot \nabla_x)\,\eta, \mu \big\rangle + \big\langle \nabla_v \eta \, \sqrt{\mu} , G \big\rangle \\
&= \big\langle (\partial_t + v \cdot \nabla_x)\,\eta, \mu \big\rangle -\tfrac{m}{m-1/2} \, \big\langle \nabla_v \eta, \nabla_v (\mu^{m-1/2}) \, \sqrt{\mu} \big\rangle - \langle \nabla_v \eta, v \, \mu \rangle \, ,
\end{align*}
which is the desired conclusion.

\section{Maximal dissipation for \texorpdfstring{$m\in [1,3/2]$}{m in [1,3/2]}}\label{sec:sec4}

Now we assume that $m\in [1,3/2]$ and fix an horizon $T>0$. We know from Section \ref{Sec:JKO} that we may assume that the family $\mu^{(h)}$ constructed by the JKO scheme converges in $\mathrm{L}^1\big((0,T)\times \Gamma\big)$ and uniformly in time in Wasserstein distance to some $\mu$. The goal of this section is to prove that $\mu$ dissipates the energy maximally.

\bigskip\noindent\textbf{Step 1: De Giorgi interpolation and discrete energy dissipation.} We first need to improve~\eqref{eq:edi0} into a sharper discrete energy dissipation inequality. To do so, following~\cite{MR2401600,MR3409718,MR3625852}, we consider a De Giorgi variational interpolation scheme. For every $k=0, \ldots, N-1$ and~$\theta \in (0,1]$, let
\[
\hat{\nu}_k^{(h)}(\theta)\coloneqq \argmin_{\nu \in \Pg} \left\{ \Fm(\nu) + \frac{1}{2\,\theta\,h} \, W_h^2 \({\Gh}{}_\# \mu^{(h)}_k, \nu\)\right\}
\]
so that~$\hat{\nu}_k^{(h)}(1) = \mu_{k+1}^{(h)}$. From
\[
\Fm\bigl(\hat \nu_k^{(h)}(\theta)\bigr) + \frac{1}{2\,\theta\,h} \, W_h^2 \({\Gh}{}_\# \mu^{(h)}_k, \hat \nu_k^{(h)}(\theta)\) \le \Fm\bigl({\Gh}{}_\# \mu^{(h)}_k\bigr)
\]
we infer that~$\theta \mapsto \nu_k^{(h)}(\theta)$ can be continuously extended (in the Wasserstein sense) at~$0$ with~$\hat{\nu}_k^{(h)}(0)= {\Gh}{}_\# \mu^{(h)}_k$, and also that
\begin{equation}\label{eq:whnu}
\sup_{\theta \in [0,1]} W_h^2 \({\Gh}{}_\# \mu^{(h)}_k, \hat{\nu}_k^{(h)}(\theta)\) \le\Cst\,.
\end{equation}

It follows from classical arguments (see~\cite[paragraph 3.2]{MR3625852} and~\cite[paragraphs 3.1-3.2]{MR2401600}) that the following dissipation inequality holds
\begin{equation}\label{eq:edi1}
\Fm\(\mu_{k+1}^{(h)}\)-\Fm\(\mu_k^{(h)}\) +\frac{1}{2\,h} \, W_h^2\({\Gh}{}_\# \mu^{(h)}_k, \mu_{k+1}^{(h)}\) + \frac{h}{2} \int_0^1 \left\vert \partial_h \Fm \(\hat{\nu}_k^{(h)}(\theta)\)\right\vert^2\,d \theta \leq 0 \, ,
\end{equation}
where
\[
\big\vert \partial_h \Fm(\mu) \big\vert\coloneqq \limsup_{\nu \to \mu} \frac{ \big(\Fm(\mu)-\Fm(\nu)\big)_+}{W_h(\mu, \nu)}
\]
denotes the descending metric slope. Here we use again $ \Fm\big(\mu_k^{(h)}\big)= \Fm \big({\Gh}{}_\# \mu^{(h)}_k\big)$, and the convergence $\nu \to \mu$ is intended in the $W_h$ sense or, equivalently, in the $W$ sense, as long as $h\in (0,1)$ is fixed.

Let us compute the metric slope. To do this, we define
\begin{equation}\label{eq:formEmh}
\Fm^{(h)}(f)\coloneqq \Fm \big({A_h^{-1}}\!_\# f\big)= \frac{\det(A_h)^{m-1}}{m-1} \iint_\Gamma f^{m}\,dx \, dv + \frac{1}{2} \iint_\Gamma \vert v\vert^2 \, f \,dx \, dv\,.
\end{equation}
Thanks to~\eqref{eq:formulasforch}, we observe that
\[
\big\vert \partial \Fm^{(h)}({A_h}{}_\#\mu) \big\vert = \limsup_{\nu \to \mu} \frac{ \(\Fm^{(h)} \({A_h}{}_\#\mu\)-\Fm^{(h)} ({A_h}{}_\#\nu)\)_+}{W\({A_h}{}_\# \mu, {A_h}{}_\#\nu\)}=\big\vert \partial_h \Fm(\mu) \big\vert
\]
where $\big\vert \partial \Fm^{(h)}\big\vert$ stands for the slope in the usual Wasserstein sense. Using the explicit form of~\eqref{eq:formEmh}, $\Fm^{(h)}$ is a displacement convex functional in the usual Wasserstein sense on $\Gamma$. Thus, when~$\big\vert \partial \Fm^{(h)}(\mu)\big\vert < +\infty$,~\cite[Theorem~10.4.13]{MR2401600} ensures that~$d\mu = f \, dx \, dv$ with~$f^m \in \mathrm{W}_\mathrm{loc}^{1,1}(\Gamma)$, and
\[
\big\vert \partial \Fm^{(h)}(\mu)\big\vert^2=\iint_\Gamma \left\vert\, \det(A_h)^{m-1} \, \frac{\nabla_{x,v} (f^m)}{f}+ (0, v) \,\right\vert^2 \, d\mu \,.
\]
{}From the change of variables $d({A_h}{}_\#\mu)=\det(A_h)^{-1} f \circ {A_h}^{-1} \, dx \, dv$, we deduce
\begin{align*}
&\big\vert \partial_h \Fm(\mu) \big\vert^2 = \big\vert \partial \Fm^{(h)}\({A_h}{}_\#\mu\)\big\vert^2\\
&= \iint_\Gamma \left\vert\, \frac{[(\nabla_{x,v} (f^m))\circ A_h^{-1}] \cdot A_h^{-1}}{f \circ A_h^{-1}}+ (0, v) \,\right\vert^2 \, d(A_{h}{}_\#\mu)\\
&= \iint_\Gamma \left\vert\, \frac{(\nabla_{x,v} (f^m)) \cdot A_h^{-1}}{f}+ (0, v) \,\right\vert^2 \, d\mu\nonumber \ge \iint_\Gamma \left\vert\, \frac{1}{f}\left(\frac{h}{2} \,\nabla_{x} + \nabla_v\right) (f^m)+ v \,\right\vert^2 \, d\mu\,.
\end{align*}
Consider~$D_h \coloneqq \frac{h}{2}\nabla_x + \nabla_v$.
When~$\big\vert \partial \Fm^{(h)}(\mu)\big\vert^2$ is finite, by Lemma~\ref{lemma:chain} applied with~$\alpha=1-1/(2\,m)$ and~$g=f^m$, we may write~$D_h (f^m)$ in terms of~$\sqrt{f}$ and~$D_h(f^{m-1/2})$, see~\eqref{eq:chain_m_mhalf}. We infer that
\[
\vert \partial_h \Fm(\mu) \vert^2 \geq \IFm^{(h)}(f) \coloneqq \iint_\Gamma \left|\,\frac{m}{m-1/2}\, D_h \bigl(f^{m-\frac{1}{2}}\bigr) + v\,\sqrt{f}\,\right|^2 \, dx \, dv \,.
\]
Moreover, by a direct change of variable argument, one can easily obtain the alternative expression for this modified Fisher information as
\begin{equation}\label{eq:modfisherh}
\IFm^{(h)}(f)=\IFm \left(f \circ \mathcal{G}_{\frac{h}{2}}\right)=\IFm \({\mathcal{G}_{-\frac{h}{2}}}{}_\# \mu\)\,.
\end{equation}
Let us  define the curve $[0,T] \ni t \mapsto \hat{\mu}^{h}_t$ by
\[
\hat{\mu}^{(h)}_t \coloneqq \hat{\nu}_k^{(h)} \Big(\frac{t-k\,h}{h}\Big)\,, \quad t\in \big(k\,h, (k+1)\,h\big]\,, \quad 0 \leq k \leq N-1\,.
\]
Inequality~\eqref{eq:edi1}, together with~\eqref{eq:formulasforch},~\eqref{eq:modfisherh}, and~\eqref{eq:metricfisher} implies
\begin{align}
0 & \geq \Fm\(\mu_{k+1}^{(h)}\)-\Fm\(\mu_{k}^{(h)}\)+ \frac{1}{2\,h} \, W_h^2\({\Gh}{}_\# \mu^{(h)}_k, \mu_{k+1}^{(h)}\) + \frac{1}{2}\int_{k\,h}^{(k+1)\,h}\kern-8pt \IFm^{(h)} \(\hat{\mu}^{(h)}_s\) ds \nonumber\\
&=\Fm\(\mu_{k+1}^{(h)}\)-\Fm\(\mu_{k}^{(h)}\)+ \frac{1}{2\,h} \, \sfd_h^2\(\mu^{(h)}_k, \mu_{k+1}^{(h)}\)+ \frac{1}{2}\int_{k\,h}^{(k+1)\,h}\kern-8pt \IFm \left({\mathcal{G}_{-\frac{h}{2}}}{}_\# \hat{\mu}^{(h)}_s\right) ds \nonumber\\
& \geq \Fm\(\mu_{k+1}^{(h)}\)-\Fm\(\mu_{k}^{(h)}\)+ \frac{1}{2} \int_{k\,h}^{(k+1)h} \! \left[ \big\Vert F^{(h)}_s \big\Vert^2_{\mathrm L^2\(\mu^{(h)}(s)\)} + \IFm \left({\mathcal{G}_{-\frac{h}{2}}}{}_\# \hat{\mu}^{(h)}_s\right) \right] ds\,.
\label{eq:discreteEDI}
\end{align}

{}From~$W^2 \leq 2 \, W_h^2$, the triangle inequality, and~\eqref{eq:basicjko1},~\eqref{eq:whnu}, we deduce
\[
\sup_{t\in[0,T]} W\(\mu^{(h)}_t, \hat{\mu}^{(h)}_t\) \le\Cst\,\sup_{t\in[0,T]} W_h\(\mu^{(h)}_t, \hat{\mu}^{(h)}_t\) \le\Cst\,\sqrt{h}.
\]
Moreover, by the very same $\mathrm L^1$-contraction and invariance by horizontal translations arguments as in Section~\ref{Sec:JKO}, the interpolation $\hat{\mu}^{(h)}$ satisfies
\begin{equation}\label{eq:L1equicontinxhatmu} 
\sup_{t\in [0,T] }\Vert {\tau_{\delta}}{}_\# \hat{\mu}^{(h)}_t - \hat{\mu}^{(h)}_t \Vert_{\mathrm L^1(\Gamma)} \leq \max_{0\leq k \leq N} \Vert {\tau_{\delta}}{}_\# \mu_k^{(h)} - \mu_k^{(h)} \Vert_{\mathrm L^1(\Gamma)} \leq \Delta_{0}(\delta).
\end{equation}

\bigskip\noindent\textbf{Step 2: Maximal dissipation.} In view of \eqref{eq:discreteEDI}, it is natural to define the family of curves of measures $\tilde{\mu}^{(h)}\coloneqq{\mathcal{G}_{-\frac{h}{2}}}{}_\# \hat{\mu}^{(h)}$. Then, we have
\begin{align*}
W^2\(\hat{\mu}^{(h)}_t, \tilde{\mu}^{(h)}_t\) &\leq \iint_\Gamma \left\vert\, \mathrm{id}-\mathcal{G}_{-\frac{h}{2}} \,\right\vert^2 \, d \hat{\mu}^{h}_t
= \frac{h^2}{4} \iint_\Gamma |v|^2 \, d\hat{\mu}^{(h)}_t \\
&\le \frac{h^2}{2} \Fm\(\hat{\mu}^{(h)}_t \) \le \Cst\,h^2 \, .
\end{align*}
As a consequence, we obtain
\begin{equation*}
\lim_{h\to 0_+} \sup_{t\in [0, T]} W\(\tilde{\mu}^{(h)}_t, \mu_t\) = 0
\end{equation*}
and we have uniform moment bounds on $\tilde{\mu}^{(h)}$. We deduce from \eqref{eq:L1equicontinxhatmu}:
\begin{equation}\label{eq:L1equicontinxtiltmu} 
\sup_{t\in [0,T] }\Vert {\tau_{\delta}}{}_\# \tilde{\mu}^{(h)}_t - \tilde{\mu}^{(h)}_t \Vert_{\mathrm L^1(\Gamma)}=\sup_{t\in [0,T] }\Vert {\tau_{\delta}}{}_\# \hat{\mu}^{(h)}_t - \hat{\mu}^{(h)}_t \Vert_{\mathrm L^1(\Gamma)} \leq \Delta_{0}(\delta).
\end{equation}
Observe now that~\eqref{eq:unif_boundLm} and~\eqref{eq:discreteEDI} implies
\begin{equation}\label{eq:boundFishertilde}
 \int_0^T \IFm \left(\tilde{\mu}^{(h)}_s\right) ds \leq \Cst.
 \end{equation}
With uniform moment bounds, \eqref{eq:L1equicontinxtiltmu} and \eqref{eq:boundFishertilde}, we can argue exactly as in Section \ref{Sec:JKO} to deduce that $\tilde{\mu}^{(h)}$ is relatively compact in $\mathrm{L}^1\big((0,T)\times \Gamma\big)$ hence we may assume
\[
 \lim_{h\to 0_+} \big\Vert \tilde{\mu}^{(h)}-\mu \big\Vert_{\mathrm{L}^1((0,T)\times \Gamma)}=0.
 \]
If $t\in (k\,h,(k+1)\,h]$, by construction $\Fm\(\mu^{(h)}_t\)= \Fm\(\mu^{(h)}_{k+1}\)$ hence summing the inequalities \eqref{eq:discreteEDI}, from $k=0$ to $\ell$, we obtain
\begin{align}
0 & \geq \Fm\(\mu^{(h)}_t\)-\Fm(\mu_0)\nonumber\\
&\qquad+ \frac{1}{2} \int_0^{(\ell+1)\,h} \Vert F^{(h)}_\tau \Vert^2_{\mathrm L^2(\mu^{(h)}_\tau)}\, d\tau+ \frac{1}{2} \int_0^{(\ell+1)\,h} \IFm \Big(\mathcal{G}_{\frac{h}{2}}{}_\# \hat{\mu}^{(h)}_\tau \Big)\, d\tau \nonumber\\
& \geq \Fm\(\mu^{(h)}_t\)-\Fm(\mu_0)+ \frac{1}{2} \int_0^{t} \Vert F^{(h)}_\tau \Vert^2_{\mathrm L^2(\mu^{(h)}_\tau)}\, d\tau+ \frac{1}{2} \int_0^{t} \IFm \Big(\ \tilde{\mu}^{(h)} \Big)\, d\tau \label{eq:edih2}\,.
\end{align}
Recalling \eqref{eq:defdeFhQh}, using \eqref{eq:boundQhL1} and a standard lower semicontinuity argument (see~\cite[Theorem 2.34]{MR1857292}), we deduce that $Q^{(h)}$ converges in the sense of distributions to some vector measure $Q$ with $Q_t=F_t \, \mu_t$, $F_t\in \mathrm L^2(\mu_t)$, such that the pair $(\mu_t,F_t)$ solves~\eqref{eq:vlasov.} and
\begin{equation}\label{eq:BBlikelsc}
\int_0^t \iint_\Gamma \vert F_\tau \vert^2\,d\mu_\tau\,d\tau= \int_0^t \iint_\Gamma \Big \vert \frac{ d Q}{ d \mu_\tau} \Big\vert^2\,d\mu_\tau\,d\tau \leq \liminf_{h \to 0_+} \int_0^t \iint_\Gamma \big\vert F^{(h)}_\tau\big\vert^2\,d\mu^{(h)}_\tau\,d\tau
\end{equation}
for any $t\in [0,T]$.

Let us now write the Fisher information of $\tilde{\mu}^{(h)}_\tau$ as
\begin{align}
 \IFm \Big(\tilde{\mu}^{(h)}_\tau\Big)=\iint_\Gamma\(\big(\tfrac{2\,m}{2\,m-1}\big)^2\,\big|\nabla_v\big(\big( \tilde{\mu}^{(h)}_\tau \big)^{m-1/2}\big)\big|^2\,+|v|^2\, \tilde{\mu}^{(h)}_\tau \)dx\,dv \label{eq:decompfisher1}\\
 + \iint_\Gamma \frac{2m}{m-1/2} \nabla_v \big(\tilde{\mu}^{(h)}_\tau \big)^{m-1/2} \cdot v\, \sqrt{\tilde{\mu}^{(h)}_\tau} dx\,dv\label{eq:decompfisher2}
\end{align}
By Lemma~\ref{lem:fisherlsc}, the first term in \eqref{eq:decompfisher1} is lower semicontinuous for the strong $\mathrm L^1$ topology.
\begin{lemma}\label{lem:fisherlsc}
Let $\alpha \geq 0$, and $(f_n)_{n\in\N}$ be a sequence of probability densities on $\Gamma$. If~$(f_n)_{n\in\N}$ converges strongly in $\mathrm L^1(\Gamma)$ to $f$, then we have
\[
\liminf_{n\to\infty} \iint_\Gamma \vert \nabla_v f_n^{\alpha} \vert^2\,dx\,dv \geq \iint_\Gamma \vert \nabla_v f^{\alpha} \vert^2\,dx\,dv\,.
\]
\end{lemma}
 In \eqref{eq:decompfisher2}, we have to handle the product of $v\, \sqrt{\tilde{\mu}^{(h)}}$ by $\nabla_v (\tilde{\mu}^{(h)})^{m-1/2}$. Note that $v\sqrt{\tilde{\mu}^{(h)}}$ converges weakly in $\mathrm{L^2}$ to $v\, \sqrt{\mu}$, but, by Wasserstein convergence, we also have convergence of second moments, which implies that the convergence of these square roots is in fact strong in $\mathrm{L^2}$. As for the $\mathrm{L^2}$-bounded sequence $\nabla_v \big(\tilde{\mu}^{(h)}\big)^{m-1/2}$, one easily deduces from the strong convergence of $\tilde{\mu}^{(h)}$ to $\mu$ in $\mathrm{L}^1$, that $\nabla_v \big(\tilde{\mu}^{(h)}\big)^{m-1/2}$ converges to $\nabla_v \big(\tilde{\mu}^{(h)}\big)^{m-1/2}$ weakly in $\mathrm{L}^2$. Using these observations and Fatou's lemma for the second term in \eqref{eq:decompfisher1}, we get
 \begin{equation}\label{eq:ineqliminffishert}
\liminf_{h\to 0_+} \int_0^t \IFm \Big(\tilde{\mu}^{(h)}_\tau\Big) \, d\tau \geq \int_0^t \IFm (\mu_\tau) \, d\tau\,.
 \end{equation}

To pass to the liminf in~\eqref{eq:edih2}, we notice that $\Fm$ is lower semicontinuous for the convergence in Wasserstein distance, so that with~\eqref{eq:BBlikelsc}, \eqref{eq:ineqliminffishert} and Fatou's lemma, we conclude that, for a.e.~$t \in [0,T]$, the limit curve satisfies
\begin{align}
0 & \geq \Fm(\mu_t)-\Fm(\mu_0)+ \frac{1}{2} \int_0^t \Vert F_\tau \Vert^2_{\mathrm L^2(\mu_\tau)}\, d\tau + \frac{1}{2} \int_0^t \IFm (\mu_\tau)\, d\tau \nonumber\\
& \geq \Fm(\mu_t)-\Fm(\mu_0)+ \frac{1}{2} \int_0^t \left(|\mu'_\tau|_\sfd^2+ \IFm (\mu_\tau)\right) d\tau \label{eq:ede3}
\end{align}
where we used the fact that $|\mu'_t|_\sfd \leq \Vert F_t\Vert_{\mathrm {L}^2(\mu_t)}$ for a.e.~$t$ by~\cite[Proposition~5.5]{brigati2025kinetic}.

For $m\in [1,3/2]$, we can now use the chain rule \eqref{eq:chainp.} in Theorem \ref{thm1.} (note that until now, we have not used any restriction on $m\geq 1$) . Indeed, fix~$t$ and a sequence~$t_n \to 0$ in the full-measure set for which~$\eqref{eq:chainp.}$ holds. Since~$t \mapsto \mu_t$ is $W$-continuous and~$\Fm$ is lower semicontinuous, we obtain
\begin{align*}
 \Fm(\mu_0) &\le \liminf_{n \to \infty} \Fm(\mu_{t_n}) \le \Fm(\mu_t) + \frac{1}{2} \liminf_{n \to \infty} \int_{t_n}^t \left(|\mu'_\tau|_\sfd^2+ \IFm (\mu_\tau)\right) d\tau \\
 &\le \Fm(\mu_t) + \frac{1}{2} \int_{0}^t \left(|\mu'_\tau|_\sfd^2+ \IFm (\mu_\tau)\right) d\tau \, .
\end{align*}
It follows that~\eqref{eq:ede3} is, in fact an equality. Writing this equality for two times~$t=a$ and~$t=b$, and by substraction, we obtain equality in~\eqref{eq:chainp.}.

\begin{acknowledgement} This work has been supported by the Project \emph{Conviviality} (JD, ANR-23-CE40-0003) and the PEPR PDE-AI project (FQ, ANR-23-PEIA-0004) of the French National Research Agency. GC and JD thank B.~Nazaret for early discussions on kinetic equations.\\
{\scriptsize\copyright\,\the\year\ by the authors. This paper may be reproduced, in its entirety, for non-commercial purposes. \href{https://creativecommons.org/licenses/by/4.0/legalcode}{CC-BY 4.0}}
\end{acknowledgement}

\end{document}